\documentclass[11pt, a4paper]{amsart}

\usepackage{tikz}
\usetikzlibrary{graphs}

\usepackage{amsmath, amsthm}

\usepackage{accents}

\newenvironment{sideremark}{\endgraf\bigskip
\indent\qquad \qquad \vrule\,\vrule \qquad \begin{minipage}{13cm}
 }
{\end{minipage}\endgraf\bigskip}

\newenvironment{itemze}{\begin{list}{$\bullet$}{%
      \setlength{\parsep}{1mm}%
      \setlength{\itemsep}{0cm}%
      \setlength{\leftmargin}{1cm}%
      \setlength{\labelwidth}{\leftmargin}%
    }{}}{\end{list}}

\usepackage[a4paper, top=3cm, left=3cm, right=2cm, bottom=2cm]{geometry}

\usepackage{amssymb}

\usepackage{amsfonts}
\usepackage{color, enumerate, fancyhdr, dsfont}

    \newcommand{\dom}{\mbox{\rm dom}}

    \newcommand{\thzfc}{\mathrm{ZFC}}

    \newcommand{\Ewf}{\mathcal{E}}
    \newcommand{\Fwf}{\mathcal{F}}
    
    \newcommand{\Hwf}{\mathcal{H}}

    \newcommand{\Mwf}{\mathcal{M}}
    \newcommand{\Nwf}{\mathcal{N}}
    
    \newcommand{\Pwf}{\mathcal{P}}

    \newcommand{\bfrak}{\mathfrak{b}}
    \newcommand{\cfrak}{\mathfrak{c}}
    \newcommand{\dfrak}{\mathfrak{d}}

    \newcommand{\xfrak}{\mathfrak{x}}
    \newcommand{\yfrak}{\mathfrak{y}}

    \newcommand{\menos}{\smallsetminus}

    \newcommand{\frestr}{\!\!\upharpoonright\!\!}

    \newcommand{\limdir}{\mbox{\rm limdir}}

\DeclareMathOperator{\supp}{supp}
\DeclareMathOperator{\cov}{cov}
\DeclareMathOperator{\add}{add}
\DeclareMathOperator{\non}{non}
\DeclareMathOperator{\cf}{cf}
\newcommand{\com}{\text{-compact}}
\DeclareMathOperator{\countable}{countable}

    \newcommand{\Aor}{\mathbb{A}}
    \newcommand{\Bor}{\mathds{B}}
    
    \newcommand{\Dor}{\mathds{D}}
    \newcommand{\Eor}{\mathds{E}}

    \newcommand{\Por}{\mathds{P}}
    \newcommand{\Qor}{\mathds{Q}}
    
    \newcommand{\Sor}{\mathds{S}}
    
    \newcommand{\Qnm}{\dot{\mathds{Q}}}

    \newcommand{\Anm}{\dot{\mathds{A}}}
    \newcommand{\Bnm}{\dot{\mathds{B}}}
    \newcommand{\Cnm}{\dot{\mathds{C}}}
    \newcommand{\Dnm}{\dot{\mathds{D}}}

    \newcommand{\Q}{\mathbb{Q}}
    \newcommand{\R}{\mathbb{R}}

    \newcommand{\sii}{{\ \mbox{$\Leftrightarrow$} \ }}

    \newcommand{\forces}{\Vdash}

    \newcommand{\stemseq}{\mathrm{SS}}

\title{The left side of Cicho\'n's diagram}

\author{Martin Goldstern}
\address{Institut f\"ur Diskrete Mathematik und Geometrie, Technische Universit\"at Wien, Wiedner Hauptstrasse 8-10/104, 1040 Vienna, Austria.}
\email{goldstern@tuwien.ac.at}
\urladdr{http://www.tuwien.ac.at/goldstern/}

\author[Diego A. Mej\'ia]{Diego Alejandro Mej\'ia}
\address{Institut f\"ur Diskrete Mathematik und Geometrie, Technische Universit\"at Wien, Wiedner Hauptstrasse 8-10/104, 1040 Vienna, Austria.}
\email{diego.guzman@tuwien.ac.at}
\urladdr{http://www.researchgate.com/profile/Diego\_Mejia2}

\author{Saharon Shelah}
\address{Einstein Institute of Mathematics, Edmond J. Safra Campus, Givat Ram, The hebrew University of Jerusalem, Jerusalem, 91904, Israel, and Department of Mathematics, Rutgers University, New Brunswick, NJ 08854, USA.}
\email{shlhetal@math.huji.ac.il}
\urladdr{http://shelah.logic.at}

\thanks{This work was
partially supported by European Research Council grant 338821.
The first and second authors were supported by the Austrian Science Fund (FWF) P24725-N25 (first author), P23875-N13 and I1272-N25 (second author) and they were
partially supported by the National Science Foundation under grant DMS-1101597.\\
Publication 1066 on Shelah's list.
}



\begin{document}

\makeatletter

\theoremstyle{plain}
  \newtheorem{theorem}{Theorem}[section]
  \newtheorem{corollary}[theorem]{Corollary}
  \newtheorem{lemma}[theorem]{Lemma}
  \newtheorem{mainlemma}[theorem]{Main Lemma}
  \newtheorem{maintheorem}[theorem]{Main Theorem}
  \newtheorem{prop}[theorem]{Proposition}
  \newtheorem{claim}[theorem]{Claim}
  \newtheorem{exer}{Exercise}
\theoremstyle{definition}
  \newtheorem{definition}[theorem]{Definition}
  \newtheorem{example}[theorem]{Example}
  \newtheorem{remark}[theorem]{Remark}
  \newtheorem{context}[theorem]{Context}
  \newtheorem{notation}[theorem]{Notation}
  \newtheorem{question}[theorem]{Question}
  \newtheorem{problem}[theorem]{Problem}
  \newtheorem{fact}[theorem]{Fact}

\begin{abstract}
Using a finite support iteration of ccc forcings, we construct
 a model of
 $\aleph_1<\add(\Nwf)<\cov(\Nwf)<\bfrak<\non(\Mwf)<\cov(\Mwf)=\cfrak$.
\end{abstract}
\maketitle

\setcounter{section}{0}

\section{Introduction}

 %

How many (Lebesgue) null sets do you need to cover the real line?
How many points do you need to get a non-null set?  What is
the smallest number of null sets that you need to get a union
which is not null any more?   The answers to these questions
are the cardinals $\cov(\Nwf)$, $\non(\Nwf)$,
$\add(\Nwf)$, and similar definitions are possible for
other ideals, such as the ideal $\Mwf$  of meager (=first category)
sets, the ideal of at most countable sets,
or the ideal of $\sigma$-compact subsets of the irrationals.

The cardinal $\add(\sigma\com)=\non(\sigma\com)$ is usually called
$\bfrak$; it is the smallest size of a family of functions from $\omega$
to $\omega$ which is not eventually bounded by a single function.  We
define $\dfrak:=\cov(\sigma\com)$, and write $\cf(I)$ for the
smallest size of a basis of any ideal $I$.

Cicho\'n's diagram
   (see \cite{CKP},
   \cite{FremCich}, \cite{BaJu95})
 is the following table of 12 cardinals:
\medskip

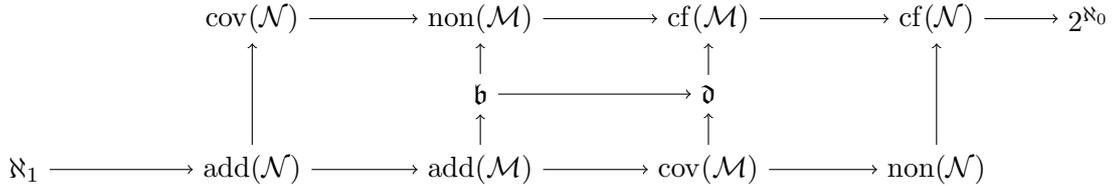
\tikz{
\node (o1) at (-1,0) {$\aleph_1$} ;
\node (an) at (2,0) {$\add(\Nwf)$} ;
\node (am) at (5,0) {$\add(\Mwf)$} ;
\node (cm) at (8,0) {$\cov(\Mwf)$} ;
\node (nn) at (11,0) {$\non(\Nwf)$} ;
\node (bb) at (5,1) {$\bfrak$} ;
\node (dd) at (8,1) {$\dfrak$} ;
\node (cn) at (2,2) {$\cov(\Nwf)$} ;
\node (nm) at (5,2) {$\non(\Mwf)$} ;
\node (fm) at (8,2) {$\cf(\Mwf)$} ;
\node (fn) at (11,2) {$\cf(\Nwf)$} ;
\node (cc) at (13,2) {$2^{\aleph_0}$} ;
\draw (o1) edge [->] (an)
      (an) edge [->] (am)
      (am) edge [->] (cm)
      (cm) edge [->] (nn)
      (nn) edge [->] (fn)  ;
      (fn) edge [->] (cc)  ;
\draw (am) edge [->] (bb)
       (bb) edge [->] (nm) ;
\draw (cm) edge [->] (dd)
       (dd) edge [->] (fm) ;
\draw (an) edge [->] (cn)
      (cn) edge [->] (nm)
      (nm) edge [->] (fm)
      (fm) edge [->] (fn)
      (fn) edge [->] (cc)   ;
\draw (bb) edge [->] (dd) ;
}

\medskip



The arrows show
provable
inequalities between these cardinals, such as
$$ \aleph_1 = \non(\countable) \le \add(\Nwf) \le \cov(\Nwf) \le 2^{\aleph_0} =  \cov(\countable).$$

In addition to the inequalities indicated it is also known that
$\add(\Mwf) = \min (\bfrak, \cov(\Mwf))$ and
$\cf(\Mwf) = \max (\dfrak, \non(\Mwf))$.

For any two of these  cardinals, say $\xfrak $ and $\yfrak$, the relation
$\xfrak\le \yfrak$ is provable in ZFC if and only if this relation can
be seen in the diagram.
However, the question how many of these cardinals can be different
in a single ZFC-universe is still open.

Some models of partial answers to this question are constructed in \cite{Me-MatIt} and \cite{FGKS}. In this paper, we will construct
a model, so far unknown, where the following strict inequalities hold:
$$  \omega_1 < \add(\Nwf) < \cov(\Nwf)  < \bfrak < \non(\Mwf) < \cov(\Mwf) = 2^{\aleph_0}$$
Moreover, the values of these cardinals can be quite arbitrary.

\section{Informal overview}\label{SecOverview}

\subsection{Increasing $\add(\Nwf)$}
Assume for simplicity that GCH holds. For any regular uncountable
cardinal  $\kappa$
there is a natural way to force~$\add(\Nwf)=\kappa$, namely a finite
support iteration $(\Por_\alpha,\Qnm_\alpha:\alpha<\kappa)$
 of length~$\kappa$, where in each step $\alpha$ the forcing $\Qnm_\alpha$
will be amoeba forcing $\Aor$, which
will add an amoeba real~$\eta_\alpha$; this real will code a null set $N_\alpha$
that covers not only all reals from $V^{\Por_\alpha}$ but even the union of
all Borel null sets whose code is in~$V^{\Por_\alpha}$.  The final model $V^{\Por_\kappa}$
will satisfy the following:
\begin{itemize}
   \item  (as $\kappa$ is regular:) Every small (i.e.: of size $<\kappa$)
    family of (Borel) null sets
   will be
   added before stage~$\kappa$; hence its union will be covered by one of
    the sets~$N_\alpha$.   So $\add(\Nwf)\ge
   \kappa$.
   \item The union of all $N_\alpha$ contains all reals and is in particular
      not of measure zero; hence  also  $\add(\Nwf)\le   \kappa$.
\end{itemize}

\newcommand{\kb}{{\kappa_{b}}}
\newcommand{\kct}{{\kappa_{ct}}}
\newcommand{\kcn}{{\kappa_{cn}}}
\newcommand{\kan}{{\kappa_{an}}}
\newcommand{\kam}{{\kappa_{am}}}
\newcommand{\knm}{{\kappa_{nm}}}
\newcommand{\n}[1]{\underaccent{\tilde}{#1}}

\newcommand{\sqpk}{\seq{p_k}{k<\omega}}
\newcommand{\sqk}[1]{\seq{#1}{k<\omega}}
\newcommand{\seq}[2]{\langle #1 \rangle_{#2}}

This model will of course also satisfy $2^{\aleph_0} =\kappa$. If we are
given two regular cardinals $\kappa_{an}$ and $\kappa_{ct}$ (we
write $\kan$ to indicate that this cardinal is intended to be the
{\bf a}dditivity of {\bf n}ull sets, and $\kct$ for the intended size of the
{\bf c}on{\bf t}inuum), then we
can construct
a ccc forcing notion $\Por$ satisfying
$$ \kappa_{an} = \add(\Nwf) < 2^{\aleph_0} = \kappa_{ct}$$
as the finite support limit of a finite
support iteration $(\Por_\alpha,\Qnm_\alpha:\alpha<\kappa_{ct})$ as follows:
\begin{itemize}
   \item For each $\alpha < \kct $ we choose a $\Por_\alpha$-name $\n X_\alpha$
      of a family of Borel  measure zero sets (or really: Borel codes of
      measure zero sets) of  size~$<\kan$.
   \item We find a (name for a) ZFC-model $M_\alpha$ of size $<\kan$ which is forced to
     include $\n X_\alpha$.   \\
      We then
     let $\Qnm_\alpha$ be the $\Por_\alpha$-name for  $\Aor\cap M_\alpha$. \\
     (So $\Qnm_\alpha$ is the $\Por_\alpha$-name for amoeba forcing in some small
     model containing~$\n X_\alpha$, where ``small'' means of size $<\kan$
     in~$V^{\Por_\alpha}$.)
    \item The generic null set $N_\alpha$ added by $\Qor_\alpha$ will cover
     the union of all measure zero sets in~$\n X_\alpha$.
\end{itemize}
If we choose the sets $\n X_\alpha$ appropriately (using a bookkeeping
argument), we can ensure that in $V^{\Por_{\kct}}$ every union of $<\kan$
null sets will be a null set; this shows that $\add(\Nwf)\ge \kan$.

The union of all null sets coded in the intermediate model
$V^{\Por_{\kan}}$ (equivalently,
the union  $\bigcup_{\alpha< \kan} N_\alpha$, where we view the $N_\alpha$
as given by Borel codes that are to be interpreted in the final model
$V^{\Por_\kct}$)  will be non-null in the final model\footnote{Another way to say this is that the reals in $\omega^\omega\cap V^{\Por_{\kan}}$ are not localized by a single slalom from $S(\omega,\Hwf)$, see Example \ref{SubsecUnbd}(4)}, witnessing
$\add(\Nwf)\le \kan$.

This method of using small subposets of classical forcing notions
is well known, see for example
  \cite{JuSh-KunenMillerchart} and \cite{Br-Cichon}.

\subsection{Increasing~$\cov(\Nwf)$,~$\bfrak$, $\non(\Mwf)$}

In a similar way we could construct a model where $\cov(\Nwf)$ is large. The
natural choice for an iterand $\Qnm_\alpha$ would be random forcing.

If we want to get $\cov(\Nwf)=\kcn < \kct = 2^{\aleph_0}$, we could use
a finite support iteration of length $\kct$ where each iterand $\Qnm_\alpha$
is the random forcing  $\Bor $ from a small submodel of the intermediate
model~$V^{P_\alpha}$. Standard bookkeeping will ensure that the
resulting model satisfies $\cov(\Nwf)\ge \kcn$.

We can also ensure that the final  model $V^{\Por_{\kct}} $ will not contain any
random reals over the intermediate model $V^{\Por_{\kcn}} $; thus we also
have $\cov(\Nwf)\le  \kcn$.

Replacing  random forcing with Hechler forcing $\Dor$, we can get a model where
the cardinal $\bfrak$ has an intermediate value.

Finally, there is a canonical forcing that will increase~$\non(\Mwf)$,
the forcing $\Eor$ which adds an  ``eventually different real''.
 Since the properties of this forcing notion will play a crucial role in our arguments, we give an
explicit definition.

\begin{definition}
   The elements of the forcing notion  $\Eor$  are pairs $p=(s,\varphi)=(s^p,\varphi^p)$
  where $s\in \omega^{<\omega}$ and there is some  $w\in \omega$ such that
  $\varphi$ is a function $\varphi: \omega \to [\omega]^{\le w}$ satisfying $s(i)\notin \varphi(i)$ for all $i\in
    \dom(s)$.    The minimal such $w$ will be called the \emph{width}
    of $\varphi$, written $w^p=\mathrm{width}(\varphi)$.

  A function $f:\omega\to\omega$ is \emph{compatible} with a condition $(s,\varphi)$ if $s$ is
  an initial segment of $f$, and $f(i)\notin \varphi(i)$ holds for all $i$.

 Our intention is that there will be a ``generic'' function $g$, such
  that each condition $p$ forces that $g$ is compatible with $p$.
  Motivated by this intention, we define $(s',\varphi')\le (s,\varphi)$
  by
\begin{itemize}
   \item  $s \subseteq s'$.
   \item  $\forall i\in \omega: \varphi(i)\subseteq \varphi'(i)$.
\end{itemize}
\end{definition}

Letting $g$ be the name for $\bigcup \{ s: (s,\varphi)\in G\}  $, the
following properties are easy to check:
\begin{remark}
  \begin{enumerate}
   \item $(s,\varphi)$ indeed forces that  $g$ is compatible with $(s,\varphi)$.
   \item If we change the definition by requiring $\varphi$ to be defined
     on $\omega\setminus \dom(s)$ only (and adding the condition
     $s'(i)\notin \varphi(i)$ in the definition of $\le_\Eor$), we get
     an equivalent forcing notion which is moreover separative.
  \item Our forcing $\Eor$ is an  inessential variant of usual ``eventually
     different'' forcing notion in \cite{Mi}.
  \end{enumerate}
\end{remark}

 %

\subsection{Putting things together}\label{SubsecIt}

Assume again GCH, and let $\aleph_1 \le \kan \le \kcn \le  \kb \le \knm \le \kct$.    We want to construct a ccc iteration $\Por$ such that $\Por$ forces
$$ \add(\Nwf) = \kan, \ \  \cov(\Nwf) = \kcn, \ \
     \bfrak = \kb, \ \
     \non(\Mwf) = \knm,  \  \
     \cov(\Mwf) =  2^{\aleph_0} = \kct.$$

A naive approach would use an iteration of length $\kct$ in which all
iterands are ``small'' versions of Amoeba forcing, random forcing,
 Hechler forcing and eventually different forcing.  Here,
\begin{itemze}
 \item ``small Amoeba'' would mean:  Amoeba forcing from
       a  model of size~$<\kan$,
 \item ``small random'' would mean:  random forcing from
     a model smaller than~$\kcn$,
 \item ``small Hechler'' would mean:  Hechler forcing from
     a model smaller than~$\kb$,
 \item ``small eventually different'' would mean:  eventually different
    forcing from
     a model smaller than~$\knm$.
\end{itemze}
If we use suitable bookkeeping, such an iteration will ensure that all
the cardinals considered are \emph{at least} their desired value.  For
example, every small family $\Fwf$ of null sets (i.e., of Borel codes
of null sets) will appear in an intermediate model, and the bookkeeping
strategy will ensure that $\Fwf$ was considered at some stage~$\alpha$. The amoeba null set added in stage $\alpha+1$ will cover all null sets coded in~$\Fwf$. Similar arguments work for the other cardinal characteristics. Moreover,
we could explicitly add Cohen reals cofinally, or use the fact that
any finite support iteration adds Cohen reals in every limit step, to
conclude that $\cov(\Mwf)\ge \kct$.

That is, the final model will satisfy
$$ \add(\Nwf) \ge  \kan, \ \  \cov(\Nwf) \ge  \kcn, \ \
     \bfrak \ge  \kb, \ \
     \non(\Mwf) \ge  \knm,  \  \
      2^{\aleph_0}\ge  \cov(\Mwf) \ge  \kct.$$
Using well-known iteration theorems (see
  \cite{JuSh-KunenMillerchart},  \cite{Br-Cichon}, \cite[Section~6]{BaJu95}, or
the summary of  \cite[Section~2]{Me-MatIt} reviewed in Section \ref{SecPresProp})
we can conclude that
\begin{itemze}
\item
 the union of the
family of null sets added in the first
$\kan$ steps still is not a null set in the final model,
\item
 there
is no random real over the model~$V^{P_{\kcn}}$,
\item
the reals from
the model $V^{P_{\knm}}$ are still nonmeager,
\item
 the iteration
does not add more than $\kct$ reals.
\end{itemze}
So we also get
$$ \add(\Nwf) \le  \kan, \ \ \cov(\Nwf) \le  \kcn, \ \
     \qquad \qquad
     \non(\Mwf) \le  \knm,  \ \
     2^{\aleph_0} \le  \kct.$$

However, it is not immediately obvious
that the reals from the model $V^{P_\kb}$
stay an unbounded
family, or more explicitly:
 that the eventually different forcing does
not add an upper bound to this family.
Indeed, it is consistent that  a small sub-po of $\Eor$ (even one of
the form
$\Eor\cap M$ for some model $M$)  adds a dominating real, see \cite{Paw-Dom}.

The full forcing $\Eor$, on the other hand,  preserves unbounded families,
see \cite{Mi}.

A variant of this construction sketched above,
where the full forcing $\Eor$ is used rather
than small subsets of $\Eor$, would preserve the unboundedness of
a $\kb$-sized family and hence guarantee $\bfrak=\kb$, at
the cost of raising the value of $\non(\Mwf)$ to $\kct$.

Another variant is described in  \cite[Theorem~3]{Me-MatIt}:  An iteration
of length  $\kct\cdot \knm$ (ordinal product) in which the full $\Eor$
forcings are used will yield a model of
   $$ \add(\Nwf) = \kan, \ \  \cov(\Nwf) = \kcn, \ \
     \bfrak = \kb, \ \
     \non(\Mwf) = \cov(\Nwf) = \knm,  \  \
     2^{\aleph_0} = \kct$$


In this paper we want to additionally get $\non(\Mwf)=\knm<\cov{\Mwf}=\cfrak=\kct$, so it seems necessary to use small subposets of $\Eor$.

The main point in the following section is to ensure that we will
preserve an unbounded family of size $\kb$ in our iteration.

\subsection{Ultrafilters help us decide}\label{SubsecUfhelps}
The actual construction that we will use will be given in
section~\ref{SecMain}.   It will be an iteration of length
$\kct\cdot\knm$ (ordinal product), where in each coordinate a
``small'' forcing is added, as described above: an amoeba forcing
of size $< \kan$, etc.

 For notational convenience we will start in a ground model where we
 already have an unbounded family $F = \{f_i:i<\kb\}$.   Moreover we
 will assume that every subfamily of size $\kb$ is again
 unbounded.

 To simplify the presentation in this section, we will consider an
 iteration adding small $\Eor$ reals only.   We will sketch how to
 construct such an iteration that does not destroy the unboundedness
 of $F$.  Adding other ``small'' forcings to the iteration
 will not be a problem, as all these forcings will be smaller
  than~$\kb$; only the small $  \Eor$ forcing notions may be of
 size~$\ge\kb$.   A detailed proof is given in
 Main Lemma~\ref{PresUnbLemma}.

 Now assume that our iteration $(\Por_\alpha,\Qnm_\alpha: \alpha <
 \delta)$
 has finite support limit~$\Por_\delta$, and that there is a
 $\Por_\delta$-name $\n g$
 of a function which bounds all~$f_\alpha$. We can find a family
 of conditions $(p_i:i<\kappa)$ and natural number $m_i$ such that
    $$p_i \forces \forall n\ge m_i: f_i(n) \le \n g(n).$$
 By thinning out our family we may assume that all $m_i$ are equal,
 and for notational simplicity we will moreover assume
  they are all~$0$.

 Moreover, we may assume that the $p_i$ form a $\Delta$-system
 satisfying a few extra uniformity  conditions (i.e., they behave
 quite uniformly on the root).

 We now choose a countable subset $i_0 < i_1 < \cdots $ of $\kappa$
 and some $\ell$ such that
 $f_{i_k}(\ell)\ge k$ for all~$k$  (this is possible, as
 otherwise our family $(f_i)_{i<\kb}$ would be bounded).
  Again assume without loss of generality~$\ell=0$, and $i_k=k$ for all~$k$.

 We now have a \emph{countable} $\Delta$-system of conditions
 $\sqpk$ in~$P_\delta$,
 where $p_k \forces \n g(0) \ge f_k(0)$ for all~$k$.

 If we can now find a $\Por_\delta$-name $\n D_\delta$
 of a non-principal ultrafilter and a condition $q$ such that
 $$ q \forces_{\Por_\delta} \{ k:  p_k \in G_\delta \} \in \n D_\delta,$$
 then we have our desired  contradiction, as already the
 empty condition forces that the set $\{ k: p_k\in G_\delta\}$
 is finite, and in fact $f_k(0)$ is bounded by the number $\n g(0)$ for any $k$ in this set.

 To get this ultrafilter $\n D_\delta$ at the end of the proof,
 we need some preparations when we set up the iteration.  The name
 $\n D_\delta$ will of course depend on the countable sequence~$(p_k)_{k<\omega}$,
 but not very much; we will partition the set of all such sequences into
 a small family $(\Lambda_\epsilon: \epsilon < \knm) $
 of sets, and for each element $\Lambda_\epsilon$ of this small family we
 will define a name for an ultrafilter that will work for all
 countable $\Delta$-systems (coded) in~$\Lambda_\epsilon$.


\subsection{Ultrafilter limits in $\Eor$}

\begin{definition}\label{def.limit}
   Let $D$ be an ultrafilter on~$\omega$.

For each sequence $\bar A=\sqk{A_k} $ of subsets of $\omega$ we
define $\lim_D \bar A \subseteq \omega$ by taking the pointwise limit of the
characteristic functions, or in other words:
$$ n \in  \lim\nolimits_D \bar A \ \Leftrightarrow \ \{ k: n\in A_k \} \in D.$$

If $\bar \varphi = \sqk{\varphi_k} $ is a sequence of slaloms (i.e.,
each $\varphi_k$ is a function from $\omega$ to $[\omega]^{<\omega}$), then
we define $\psi:=\lim_D \bar \varphi$ as the function with domain $\omega$
satisfying $$ \psi(n) = \lim\nolimits_D \sqk{\varphi_k(n)}  $$
\end{definition}

In general the ultrafilter limit of a sequence of slaloms may contain
infinite sets.  However, the following fact gives a sufficient
condition for bounding the size of the sets in the ultrafilter limit.
\begin{fact}
  If $\bar A = \sqk{A_k}$ is a sequence of subsets of $\omega$, $b\in \omega$,
   and all
  $A_k$ satisfy $|A_k| \le b$, then also $\lim_D \bar A$ will have
  cardinality at most~$b$.

  Similarly, if $\bar \varphi = \sqk{\varphi_k}$ is a sequence of slaloms
  with the property that
  there is a number $b$ with $|\varphi_k(n)|\le b$ for all~$k$,~$n$, then
  also $\lim_D\bar \varphi$ will be a slalom consisting of sets of size $\le b$.
\end{fact}

\begin{definition}
   Let $s\in\omega^{<\omega}$, $w<\omega$, $D$ a non-principal ultrafilter on $\omega$ and $\bar p=\sqpk$ a sequence of conditions in $\Eor$ where $p_k=(s,\varphi_k)$ for some slalom $\varphi_k$ of width $\leq w$. $\lim_D\bar p$, \emph{the $D$-limit of $\bar p$ in $\Eor$}, is defined as the condition $(s,\lim_D\sqk{\varphi_k})$.
\end{definition}

To explain the connection between a sequence $\bar p=\sqpk$ and its
ultrafilter limit, we point out the following fact.   A stronger
version will be proved in Claim~\ref{succClaim}.
\begin{fact}\label{evDiffFact}
     Let $M$ be a small model.
     Let $D$ be an ultrafilter with $D\cap M \in M$,
     let $\Qor = \Eor \cap M$, $s\in \omega^{<\omega}$ and let $m^*<\omega$.

     Let $\bar \varphi=  \sqk{\varphi_k} $ be a sequence
     of slaloms of width bounded by $m^*$ and assume
     $\bar \varphi \in M$.

     Then the $D$-limit $q$ of the sequence $\bar p = \sqpk = \sqk{(s, \varphi_k)}$
     satisfies
     \begin{itemize}
       \item $q\in \Eor \cap M$.
       \item $q$ forces in $\Qor$ that the set
          $\{k<\omega \ / \ p_k\in G \}$   is infinite.
     \end{itemize}
\end{fact}
\begin{proof}
   It is clear that $q\in \Eor$.   Since $M$ contains both the sequence
   $\bar p$ and the set $D\cap M$, we can compute $\lim_D \bar p $
   in $M$, hence $q\in M$.

   Now  assume that some $q'\le q$ forces that
          $\{k<\omega \ / \ p_k\in G \}$   is bounded by some $k_*$,
   so $q$ is incompatible with all $p_k$, $k>k_*$.

   For  each $i\in \dom(s^{q'}) $ we have $s(i)\notin \varphi^q(i)$,
   so the set $B_i:= \{k<\omega \ / \ s(i)\notin \varphi_k(i)\}$ is in $D$.
   Let $k\in \bigcap_{i\in \dom(s^{q'})} B_i $ be larger than $k_*$.
   Then $q'$ and $p_k$ are compatible.
\end{proof}

\section{Background on preservation properties}\label{SecPresProp}

For reader's convenience, we recall the preservation properties summarized in \cite[Sect. 2]{Me-MatIt} which will be applied in the proof of the Main Theorem \ref{MainThm}. These preservation properties were developed for fsi of ccc posets by Judah and Shelah \cite{JuSh-KunenMillerchart}, with improvements by Brendle \cite{Br-Cichon}. These are summarized and generalized in \cite{gold} and in \cite[Sect. 6.4 and 6.5]{BaJu95}.

\begin{context}\label{ContextUnbd}
 Fix an increasing sequence $\langle\sqsubset_n\rangle_{n<\omega}$ of 2-place closed relations (in the topological sense) in $\omega^\omega$ such that, for any $n<\omega$ and $g\in\omega^\omega$, $(\sqsubset_n)^g=\left\{f\in\omega^\omega\ /\ f\sqsubset_n g\right\}$ is (closed) nwd (nowhere dense).

 Put $\sqsubset=\bigcup_{n<\omega}\sqsubset_n$. Therefore, for every $g\in\omega^\omega$, $(\sqsubset)^g$ is an $F_\sigma$ meager set.

 For $f,g\in\omega^\omega$, say that \emph{$g$ $\sqsubset$-dominates $f$} if $f\sqsubset g$. $F\subseteq\omega^\omega$ is a \emph{$\sqsubset$-unbounded family} if no function in $\omega^\omega$ $\sqsubset$-dominates all the members of $F$. Associate with this notion the cardinal $\bfrak_\sqsubset$, which is the least size of a $\sqsubset$-unbounded family. Dually, say that $C\subseteq\omega^\omega$ is a \emph{$\sqsubset$-dominating family} if any real in $\omega^\omega$ is $\sqsubset$-dominated by some member of $C$. The cardinal $\dfrak_\sqsubset$ is the least size of a $\sqsubset$-dominating family.

 Given a set $Y$, say that a real $f\in\omega^\omega$ is \emph{$\sqsubset$-unbounded over $Y$} if $f\not\sqsubset g$ for every $g\in Y\cap\omega^\omega$.
\end{context}

It is clear that $\bfrak_\sqsubset\leq\non(\Mwf)$ and $\cov(\Mwf)\leq\dfrak_\sqsubset$.

Context \ref{ContextUnbd} is defined for $\omega^\omega$ for simplicity, but in general, the same notions apply by changing the space for the domain or the codomain of $\sqsubset$ to another uncountable Polish space whose members can be coded by reals in $\omega^\omega$.

From now on, fix $\theta$ an uncountable regular cardinal.

\begin{definition}[Judah and Shelah {\cite{JuSh-KunenMillerchart}}, {\cite[Def. 6.4.4]{BaJu95}}]\label{DefPresProp}
   A forcing notion $\Por$ is \emph{$\theta$-$\sqsubset$-good} if the following property holds\footnote{\cite[Def. 6.4.4]{BaJu95} has a different formulation, which is equivalent to our formulation for $\theta$-cc posets (recall that $\theta$ is uncountable regular). See \cite[Lemma 2]{Me-MatIt} for details.}: For any $\Por$-name $\dot{h}$ for a real in $\omega^\omega$, there exists a nonempty $Y\subseteq\omega^\omega$ (in the ground model) of size $<\theta$ such that, for any $f\in\omega^\omega$ $\sqsubset$-unbounded over $Y$, $\Vdash f\not\sqsubset\dot{h}$.

   Say that $\Por$ is \emph{$\sqsubset$-good} if it is $\aleph_1$-$\sqsubset$-good.
\end{definition}

This is a standard property associated to preserve $\bfrak_\sqsubset$ small and $\dfrak_\sqsubset$ large through forcing extensions that have the property. $F\subseteq\omega^\omega$ is \emph{$\theta$-$\sqsubset$-unbounded} if, for any $X\subseteq\omega^\omega$ of size $<\theta$, there exists an $f\in F$ which is $\sqsubset$-unbounded over $X$. It is clear that, if $F$ is such a family, then $\bfrak_\sqsubset\leq|F|$ and $\theta\leq\dfrak_\sqsubset$. On the other hand, $\theta$-$\sqsubset$-good posets preserve, in any generic extension, $\theta$-$\sqsubset$-unbounded families of the ground model and, if $\lambda\geq\theta$ is a cardinal and $\dfrak_\sqsubset\geq\lambda$ in the ground model, then this inequality is also preserved in any generic extension (see, e.g., \cite[Lemma 6.4.8]{BaJu95}). It is also known (from \cite{JuSh-KunenMillerchart}) that the property of Definition \ref{DefPresProp} is preserved under fsi of $\theta$-cc posets. Also, for posets $\Por\lessdot\Qor$, if $\Qor$ is $\theta$-$\sqsubset$-good, then so is $\Por$.

\begin{lemma}[{\cite[Lemma 4]{Me-MatIt}}]\label{smallPlus}
   Any poset of size $<\theta$ is $\theta$-$\sqsubset$-good. In particular, Cohen forcing is $\sqsubset$-good.
\end{lemma}

\begin{example}\label{SubsecUnbd}
 \begin{enumerate}[(1)]
  \item \emph{Preserving non-meager sets:} For $f,g\in\omega^\omega$ and $n<\omega$, define $f\eqcirc_n g$ iff $\forall_{k\geq n}(f(k)\neq g(k))$, so $f\eqcirc g$ iff $f$ and $g$ are eventually different, that is, $\forall^\infty_{k<\omega}(f(k)\neq g(k))$. Recall form \cite[Thm. 2.4.1 and 2.4.7]{BaJu95} that $\bfrak_\eqcirc=\non(\Mwf)$ and $\dfrak_\eqcirc=\cov(\Mwf)$.

  \item \emph{Preserving unbounded families:} For $f,g\in\omega^\omega$, define $f\leq^*_n g$ iff $\forall_{k\geq n}(f(k)\leq g(k))$, so $f\leq^*g$ iff $\forall^\infty_{k\in\omega}(f(k)\leq g(k))$. Clearly, $\bfrak=\bfrak_{\leq^*}$ and $\dfrak=\dfrak_{\leq^*}$. Miller \cite{Mi} proved that $\Eor$ is $\leq^*$-good. Random forcing $\Bor$ is also $\leq^*$-good because it is $\omega^\omega$-bounding. But, as discussed in Section \ref{SecOverview}, subposets of both may add dominating reals.

  \item \emph{Preserving null-covering families:} Let $\langle I_k\rangle_{k<\omega}$ be the interval partition of $\omega$ such that $|I_k|=2^{k+1}$ for all $k<\omega$. For $n<\omega$ and $f,g\in2^\omega$ define $f\pitchfork_ng\sii\forall_{k\geq n}(f\frestr I_k\neq g\frestr I_k)$, so $f\pitchfork g\sii \forall^\infty_{k<\omega}(f\frestr I_k\neq g\frestr I_k)$. Clearly, $(\pitchfork)^g$ is a co-null $F_\sigma$ meager set. This relation is related to the covering-uniformity of measure because $\cov(\Nwf)\leq\bfrak_\pitchfork$ and $\dfrak_\pitchfork\leq\non(\Nwf)$ (see \cite[Lemma 7]{Me-MatIt}).

  It is known from \cite[Lemma $1^*$]{Br-Cichon} that, given an infinite cardinal $\nu<\theta$, every $\nu$-centered forcing notion is $\theta$-$\pitchfork$-good.

  \item \emph{Preserving ``union of non-null sets is non-null'':} Fix $\Hwf:=\{id^{k+1}\ /\ k<\omega\}$ (where $id^{k+1}(i)=i^{k+1}$) and let $S(\omega,\Hwf)=\{\psi:\omega\to[\omega]^{<\omega}\ /\ \exists_{h\in\Hwf}\forall_{i<\omega}(|\psi(i)|\leq h(i))\}$. For $n<\omega$, $x\in\omega^\omega$ and a slalom $\psi\in S(\omega,\Hwf)$, put $x\in^*_n\psi$ iff $\forall_{k\geq n}(x(k)\in\psi(k))$, so $x\in^*\psi$ iff $\forall^\infty_{k<\omega}(x(k)\in\psi(k))$. By Bartoszy\'{n}ski's characterization \cite[Thm. 2.3.9]{BaJu95} applied to $id$ and to a function $g$ that dominates all the functions in $\Hwf$, $\add(\Nwf)=\bfrak_{\in^*}$ and $\cf(\Nwf)=\dfrak_{\in^*}$.

  Judah and Shelah \cite{JuSh-KunenMillerchart} proved that, given an infinite cardinal $\nu<\theta$, every $\nu$-centered forcing notion is $\theta$-$\in^*$-good. Moreover, as a consequence of results of Kamburelis \cite{kamburelis}, any subalgebra\footnote{Here, $\Bor$ is seen as the complete Boolean algebra of Borel sets (in $2^\omega$) modulo the null ideal.} of $\Bor$ is $\in^*$-good.

 \end{enumerate}
\end{example}

For a relation $\sqsubset$ as in Context \ref{ContextUnbd}, the following practical results present facts about adding Cohen reals that form strong $\sqsubset$-unbounded families.

\begin{lemma}\label{unbCohen}
   Let $\langle\Por_\alpha\rangle_{\alpha\leq\theta}$ be a $\lessdot$-increasing sequence of ccc posets where $\Por_\theta=\limdir_{\alpha<\theta}\Por_\alpha$. Assume that $\Por_{\alpha+1}$ adds a Cohen real $\dot{c}_\alpha$ over $V^{\Por_\alpha}$ for all $\alpha<\theta$. Then, $\Por_\theta$ forces that $\{\dot{c}_\alpha\ /\ \alpha<\theta\}$ is a $\theta$-$\sqsubset$-unbounded family.
\end{lemma}

\begin{corollary}\label{larged}
   Let $\delta\geq\theta$ be an ordinal and $\Por_\delta=\langle\Por_\alpha,\Qnm_\alpha\rangle_{\alpha<\delta}$ a fsi such that, for $\alpha<\delta$,
   \begin{enumerate}[(i)]
      \item $\Por_\alpha$ forces that $\Qnm_\alpha$ is $\theta$-$\sqsubset$-good and
      \item when $\alpha<|\delta|$, $\Por_{\alpha+1}$ adds a Cohen real over $V^{\Por_\alpha}$.
   \end{enumerate}
   Then, $\Por_\delta$ forces
   \begin{enumerate}[(a)]
     \item $\bfrak_\sqsubset\leq\theta$ and
     \item $\dfrak_\sqsubset\geq|\delta|$.
   \end{enumerate}

\end{corollary}
\begin{proof}
   By (ii) and Lemma \ref{unbCohen}, for any $\nu\in[\theta,|\delta|]$ regular, $\Por_\nu$ adds a $\nu$-$\sqsubset$-unbounded family of size $\nu$ (of Cohen reals), which is preserved to be $\nu$-$\sqsubset$-unbounded in $V^{\Por_\delta}$ by (i). Therefore, $\Por_\delta$ forces $\bfrak_\sqsubset\leq\nu\leq\dfrak_\sqsubset$ for any regular $\nu\in[\theta,|\delta|]$, so $\bfrak_\sqsubset\leq\theta$ and $|\delta|\leq\dfrak_\sqsubset$.
\end{proof}

\section{Iteration candidates}\label{SecItCandidates}

We describe, in a general way, the type of iterations and the characteristics and elements it may have in order to preserve unbounded families of a certain size. Fix, in this section, an uncountable regular cardinal $\kb$ (which represents the size of an unbounded family we want to preserve).

For our main result (Theorem~\ref{MainThm}), as described in the introduction, we may use a fsi alternating between small ccc posets and subposets of $\Eor$ and find an iteration where we can preserve an unbounded family of a desired size ($\kb$). We describe, in general, such iterations as follows.

\begin{definition}\label{DefItcand}
   An \emph{iteration candidate} $\mathbf{q}$ consists of
   \begin{enumerate}[(i)]
     \item an ordinal $\delta_{\mathbf{q}}$ (the length of the iteration) partitioned into two sets $S_{\mathbf{q}}$ and $C_{\mathbf{q}}$ (the first set represent the coordinates where a subposet of $\Eor$ is used, while the second set corresponds to the coordinates where small ccc posets are used),
     \item ordinals $\langle Q_{\mathbf{q},\alpha}\rangle_{\alpha\in C_{\mathbf{q}}}$ less that $\kb$ (the domains of the small ccc posets),
     \item a fsi $\langle\Por_{\mathbf{q},\alpha},
         \Qnm_{\mathbf{q},\alpha}\rangle_{\alpha<\delta_{\mathbf{q}}}$ and a sequence $\langle\Por'_{\mathbf{q},\alpha}\rangle_{\alpha\in S_{\mathbf{q}}}$ such that
         \begin{itemize}
           \item for $\alpha\in S_{\mathbf{q}}$, $\Por'_{\mathbf{q},\alpha}\lessdot\Por_{\mathbf{q},\alpha}$ and $\Qnm_{\mathbf{q},\alpha}$ is a $\Por'_{\mathbf{q},\alpha}$-name for $\Eor^{V^{\Por'_\alpha}}$ and
           \item for $\alpha\in C_{\mathbf{q}}$, $\Qnm_{\mathbf{q},\alpha}$ is a $\Por_{\mathbf{q},\alpha}$-name of a ccc poset whose domain is $Q_{\mathbf{q},\alpha}$.
         \end{itemize}
   \end{enumerate}
   The subindex $\mathbf{q}$ may be omitted when it is obvious from the
   context.\\
   For each $\alpha\leq\delta$, consider the set $\Por^*_{\mathbf{q},\alpha}=\Por^*_\alpha$ of conditions
   $p\in\Por_\alpha$ that satisfy:
   \begin{itemize}
   \item
   if $\xi\in\supp(p)\cap C$ then
   $p(\xi)$ is an ordinal in $Q_\alpha$ (not just a name)
   \item  if
   $\xi\in\supp(p)\cap S$, then $p(\xi)$ is of the form $(s,\dot{\varphi})$
   where $s\in\omega^{<\omega}$ (not just a name), $\dot{\varphi}$ is a
   $\Por'_\xi$-name of a slalom (not just a $\Por_\xi$-name of a slalom in $V^{\Por'_\xi}$) and $\mathrm{width}(\dot{\varphi})$ is decided,
   that is, there is an $n<\omega$ such that
   $\Vdash_{\Por'_\xi}n=\mathrm{width}(\dot{\varphi})$
   \end{itemize}
   It is easy to prove  (by induction on $\alpha$) that $\Por^*_\alpha$ is dense in $\Por_\alpha$.

   For $\alpha\leq\delta_{\mathbf{q}}$, $\mathbf{q}\frestr\alpha$ denotes the iteration $\mathbf{q}$ restricted up to $\alpha$, so $\delta_{\mathbf{q}\upharpoonright\alpha}=\alpha$. Clearly, $\mathbf{q}\frestr\alpha$ is an iteration candidate.
\end{definition}

The beginning of the proof of Main Lemma~\ref{PresUnbLemma} shows a typical argument with a $\Delta$-system to prove that an iteration candidate preserves an unbounded family of size $\kb$ (as sketched in Subsection~\ref{SubsecUfhelps}). Therefore, in order to extend Miller's compactness argument \cite{Mi} to fsi, we start by coding the relevant elements of countable $\Delta$-systems of iteration candidates by \emph{stem sequences}, as it is described below.

\begin{definition}\label{DefBlueprint}
   Let $\alpha^*$ be an ordinal. A \emph{stem sequence} $\mathbf{x}\in\stemseq_{\alpha^*}$ (of a countable $\Delta$-system) consists of
   \begin{enumerate}[(i)]
     \item a countable set of ordinals $w_{\mathbf{x}}\subseteq\alpha^*\cup\kb$ (where the relevant information of the coded $\Delta$-system lives),
     \item a natural number $l^*_{\mathbf{x}}$ (the size of the support of the conditions in the coded $\Delta$-system) partitioned into two sets $v_{\mathbf{x},S}$ and $v_{\mathbf{x},C}$ (the first set indicate the position of coordinates where a subposet of $\Eor$ is used, while the second set corresponds to the positions where small ccc posets are used),
     \item a subset $v_{\mathbf{x}}$ of $l^*_{\mathbf{x}}$ (the set of positions of the coordinates of the root of the $\Delta$-system),
     \item a subset $\{\alpha_{\mathbf{x},k,l}\ /\ k<\omega,\
       l<l^*_{\mathbf{x}}\}$ of $w_{\mathbf{x}}\cap\alpha^*$ satisfying:
       $\langle\{\alpha_{\mathbf{x},k,l}\ /\
       l<l^*_{\mathbf{x}}\}\rangle_{k<\omega}$ is a $\Delta$-system with root
       $\Delta_{\mathbf{x}}=\{\alpha^*_{\mathbf{x},l}\ /\ l\in v_{\mathbf{x}}\}$
       where, for $l\in v_{\mathbf{x}}$ and $k<\omega$,
       $\alpha_{\mathbf{x},k,l}=\alpha^*_{\mathbf{x},l}$; moreover, $\{\alpha_{\mathbf{x},k,l}\ /\  l<l^*\}$ is an increasing enumeration for each $k<\omega$ and $\langle\alpha_{\mathbf{x},k,l}\ /\ k<\omega\rangle$ is increasing\footnote{This is only needed for the proof of Claim \ref{limitClaim}.} for each $l\in l^*_{\mathbf{x}}\menos v_{\mathbf{x}}$,
     \item ordinals $\langle\gamma_{\mathbf{x},k,l}\ /\ k<\omega,\ l\in v_{\mathbf{x},C} \rangle$ (the sequence of ordinals used at the $l$-th position of the $k$-th condition of the $\Delta$-system) and $\langle\gamma^*_{\mathbf{x},l}\ /\ l\in v_{\mathbf{x}}\cap v_{\mathbf{x},C}\rangle$ in $\kb\cap w_{\mathbf{x}}$ such that $\gamma_{\mathbf{x},k,l}=\gamma^*_{\mathbf{x},l}$ for all $l\in v_{\mathbf{x}}\cap v_{\mathbf{x},C}$ and $k<\omega$ (that is, the ordinals used at the positions of the root are the same for all $k$),
     \item a sequence $\langle s^*_{\mathbf{x},l}\ /\ l\in v_{\mathbf{x},S}\rangle$ of objects from $\omega^{<\omega}$ (the sequence of stems used at the $l$-th position of a condition, which is the same for all the conditions in the $\Delta$-system) and
     \item a sequence $\bar{n}^*_{\mathbf{x}}=\langle n^*_{\mathbf{x},l}\ /\ l\in v_{\mathbf{x},S}\rangle$ of natural numbers (the sequences of widths of slaloms at the $l$-th position of a condition in the $\Delta$-system).
   \end{enumerate}
   When there is no place to confusion, we may omit the subindex $\mathbf{x}$ for the objects of a stem sequence.
\end{definition}

If $\mathbf q$ is an iteration candidate of length $\delta$,
then every sufficiently uniform countable $\Delta$-system $\bar p =
\sqpk$ from $\Por^*_\delta$ will define a stem sequence.  But not every stem sequence is
realized
by some sequence of conditions from $\Por^*_\delta$.  In the next definition
we give a sufficient condition for a stem sequence to be realized by
an iteration, and we explain how this stem sequence gives partial information
about a sequence of conditions.


\begin{definition}\label{DefLegalBlueprint}
   A stem sequence $\mathbf{x}\in\stemseq_{\alpha^*}$ (as in Definition~\ref{DefBlueprint}) is \emph{legal} for an iteration candidate $\mathbf{q}$ (as in Definition~\ref{DefItcand}) if the following hold for each $k<\omega$ and $l<l^*$ such that $\alpha_{k,l}<\delta=\delta_{\mathbf{q}}$:
   \begin{enumerate}[(i)]
      \item $\alpha_{k,l}\in C$ iff $l\in v_C$.
      \item If $l\in v_C$ then $\gamma_{k,l}<Q_{\alpha_{k,l}}$.
   \end{enumerate}
   In this case, define $P^\infty_{\mathbf{q},\mathbf{x}}$ the set of sequences $\sqpk$ of conditions  in $\Por^*_\delta$ such that
   \begin{itemize}
     \item $\supp(p_{k})=\{\alpha_{k,l}\ /\ l<l^*\}\cap\delta$,
     \item if $\xi=\alpha_{k,l}\in\supp(p_k)\cap C$ then $p_{k}(\xi)=\gamma_{k,l}$ and
     \item if $\xi=\alpha_{k,l}\in\supp(p_k)\cap S$ then $p_{k}(\xi)=(s^*_l,\dot{\varphi}_{k,l})$ where $\dot{\varphi_{k,l}}$ is a $\Por'_\xi$-name of a slalom of width $n^*_l$.
   \end{itemize}
   Note that $\langle\supp(p_{k})\rangle_{k<\omega}$ forms a $\Delta$-system.

   Here, $P^\infty_{\mathbf{q},\mathbf{x}}$ is the set of countable $\Delta$-system that \emph{matches} the stem sequence $\mathbf{x}$.

   If $\mathbf{x}$ is legal for $\mathbf{q}$ and $\alpha\leq\delta$, then $\mathbf{x}$ is legal for $\mathbf{q}\frestr\alpha$ and, for any $\sqpk\in P^\infty_{\mathbf{q},\mathbf{x}}$, we also have $\sqk{p_{k}\frestr\alpha}\in P^\infty_{\mathbf{q}\upharpoonright\alpha,\mathbf{x}}$.
\end{definition}

Note that a stem sequence has full information about the conditions it represents only on the set $C$ where we use the small forcings.   On the set $S$, the stem sequence only knows the stems of conditions in $\Eor$, not the slaloms.

The main idea of preserving unbounded families in $\Eor$ is that, given a sequence of conditions $\sqpk$ in $\Eor$ that agree in the stems and in the width of the slaloms, it is possible to construct a limit $q$ of the sequence such that $q$ forces that infinitely many $p_k$ belong to the generic set (see Fact \ref{evDiffFact}). This limit can be found by an ultrafilter limit on the slaloms from an ultrafilter $D$ in the ground model. Moreover, there is an ultrafilter in the extension that contains $D$ as well as all sets of the form
$A_{\bar p}:=\{k<\omega\ /\ p_k\in G\}$
($G$ is the $\Eor$-generic filter) for such a sequence
$\bar p=\sqpk$ with limit $q$ in $G$.
To extend this argument to an iteration candidate, we define a kind of ultrafilter limit of a countable $\Delta$-system that matches a given stem sequence.

\begin{definition}\label{DefNiceCandidate}
   Let $\mathbf{q}$ be an iteration candidate and $\mathbf{x}\in\stemseq_{\alpha^*}$ a legal stem sequence for $\mathbf{q}$. Say that \emph{$\mathbf{D}=\langle\dot{D}_\alpha\rangle_{\alpha\leq\delta}$ solves $\mathbf{x}$ (with respect to $\mathbf{q}$)} if the following holds for each $\alpha\leq\delta$.
   \begin{enumerate}[(1)]
      \item $\dot{D}_\alpha$ is a $\Por_\alpha$-name for a non-principal ultrafilter on $\omega$.
      \item For $\alpha\in S$, $\Vdash_{\alpha}\dot{D}_\alpha\cap\Pwf(\omega)^{V^{\Por'_\alpha}}\in V^{\Por'_\alpha}$.
      \item $\alpha<\beta\leq\delta$ implies $\Vdash_{\beta}\dot{D}_\alpha\subseteq\dot{D}_\beta$.
      \item If $\alpha$ contains $\Delta_{\mathbf{x}}\cap\delta$ and $\sqpk\in P^\infty_{\mathbf{q},\mathbf{x}}$, then $q\Vdash_{\alpha}\{k<\omega\ /\ p_{k}\frestr\alpha\in\dot{G}_\alpha\}\in\dot{D}_\alpha$ where $q=\lim_{\mathbf{D}}p_k$, the \emph{$\mathbf{D}$-limit of $\sqpk$}, is defined as
          \begin{enumerate}[(i)]
            \item $\supp(q)=\Delta_{\mathbf{x}}\cap\delta$,
            \item if $\xi=\alpha_l^*\in\supp(q)\cap C$ then $q(\xi)=\gamma^*_l$ and
            \item if $\xi=\alpha_l^*\in\supp(q)\cap S$ then $q(\xi)=(s^*_l,\dot{\psi}_{l})$ where $\dot{\psi}_{l}$ is a $\Por_\xi$-name of the $\dot{D}_\xi$-limit of $\langle\dot{\varphi}_{k,l}\rangle_{k<\omega}$ (here, $p_k(\xi)=(s^*_l,\dot{\varphi}_{k,l})$ for each $k<\omega$).
          \end{enumerate}
          As each $\dot{\varphi}_{k,l}$ is a $\Por'_\xi$-name (because $p_k\in\Por^*_\delta$), by (2), $\Por_\xi$ forces $\dot{\psi}_l\in V^{\Por'_\xi}$ and $q(\xi)\in\Eor\cap V^{\Por'_\xi}$. Therefore, $q$ is a condition in $\Por_\alpha$.
   \end{enumerate}

   Say that $\mathbf{q}$ is a \emph{nice iteration candidate} if any $\mathbf{x}\in\stemseq_\delta$ (with $\delta=\delta_{\mathbf{q}}$) legal for $\mathbf{q}$ can be solved by some $\mathbf{D}$.
\end{definition}

\begin{remark}
  In (4)(iii) of Definition \ref{DefNiceCandidate}, if $\dot{\varphi}_{k,l}$ were just a $\Por_\xi$-name of a slalom in $V^{\Por'_\xi}$ for each $k<\omega$, we would not be able to guarantee that $\langle\dot{\varphi}_{k,l}\ /\ k<\omega\rangle$ is a sequence in $V^{\Por'_\xi}$, so the ultrafilter limit $\psi_l$ and $q(\xi)$ may not be in $V^{\Por'_\xi}$.

  On the other hand, in (4), $q\in\Por_\alpha$ but it may not be a condition in $\Por^*_\alpha$ because, in (iii), $\psi_l$ may not be a $\Por'_\xi$-name. However, for the nice iteration candidate constructed in Theorem \ref{MainThm}, there is a $\Por'_\alpha$-name of $\dot{D}_\alpha\cap V^{\Por'_\alpha}$ for each $\alpha\in S$, which guarantees that, in (4), $q\in\Por^*_\alpha$.
\end{remark}


\begin{mainlemma}\label{PresUnbLemma}
   Let $B=\{f_\eta\ /\ \eta<\kb\}$ be an unbounded family such that, for any $K\in[\kb]^\kb$, the set  $B\frestr K:=\{f_\eta\ /\ \eta\in K\}$ is unbounded. Then, any nice iteration candidate preserves the unboundedness of  $B$.
\end{mainlemma}
\begin{proof}
   Let $\mathbf{q}$ be a nice iteration candidate as in Definitions~\ref{DefItcand} and \ref{DefNiceCandidate}. Towards a contradiction, let $p\in\Por_\delta$ and $\dot{g}$ be a $\Por_\delta$-name for a real such that $p$ forces that $\dot{g}$ dominates all the functions in $B$. For each $\eta<\kb$ choose $m_\eta<\omega$ and $p_\eta\leq p$ in $\Por^*_\delta$ such that $p_\eta\Vdash_{\delta}\forall_{m\geq m_\eta}(f_\eta(m)\leq\dot{g}(m))$. Let $u_\eta=\supp(p_\eta)$. By the $\Delta$-system lemma, we can find $K\in[\kb]^\kb$ such that $\{u_\eta\ /\ \eta\in K\}$ forms a $\Delta$-system. Moreover, we may assume:
   \begin{enumerate}[(a)]
      \item There is an $m^*$ such that $m_\eta=m^*$ for all $\eta\in K$.
      \item There is an $l^*$ such that $u_\eta=\{\alpha_{\eta,l}\ /\ l<l^*\}$ (increasing enumeration) for all $\eta\in K$.
      \item There is a $v\subseteq l^*$ such that the root of the $\Delta$-system is $\{\alpha_{\eta,l}\ /\ l\in v\}$ for any $\eta\in K$.
      \item For each $l<l^*$ with $l\notin  v$, $\seq{\alpha_{\eta,l}}{\eta\in K}$ is increasing.
      \item There is a $v_S\subseteq l^*$ such that $\supp(p_\eta)\cap S=\{\alpha_{\eta,l}\ /\ l\in v_S\}$ for all $\eta\in K$.
      \item For each $l\in v\menos v_S$ there is an ordinal $\gamma^*_{l}$ such that $p_\eta(\alpha_{\eta,l})=\gamma_{l}^*$ for all $\eta\in K$.  (Why? Recall that the forcing notions $\Q_\alpha$ with $\alpha\in C$ live on sets $Q_\alpha$ of cardinality $<\kb$.)
      \item For each $l\in v_S$ there is an $s^*_l\in\omega^{<\omega}$ and an $n^*_l<\omega$ such that $p(\alpha_{\eta,l})$ is of the form $(s^*_l,\dot{\varphi}_{\eta,l})$ for all $\eta\in K$, where $\dot{\varphi_{\eta,l}}$ is a $\Por'_\alpha$-name for a slalom of width $n^*_l$.
   \end{enumerate}
   In the ground model, we can find
   an increasing sequence $\langle\eta_k\rangle_{k<\omega}$ in $K$ and an $m\geq m^*$ such that $\langle f_{\eta_k}(m)\ /\ k<\omega\rangle$ is increasing. This is because there is $m\geq m^*$ and infinitely many $a\in\omega$ such that $\{\eta\in K\ /\ f_\eta(m)=a\}$ has size $\kb$ (if this were not the case, then there is a $K'\in[K]^\kb$ such that $B\frestr K'$ is bounded, which contradicts the hypothesis).

   Now it is easy to find a legal stem sequence $\mathbf{x}\in\stemseq_\delta$
   for $\mathbf{q}$ such that $\bar p:=\sqk{p_{\eta_k}}\in
   P^\infty_{\mathbf{q},\mathbf{x}}$, so there is some
   $\mathbf{D}=\langle\dot{D}_\alpha\rangle_{\alpha\leq\delta}$ solving
   $\mathbf{x}$ (as in Definition~\ref{DefNiceCandidate}). Let
   $q=\lim_{\mathbf{D}} \bar p \in\Por_\delta$, so $$q\Vdash_\delta\{k<\omega\ /\
   p_{\eta_k}\in\dot{G}\}\in\dot{D}_\delta,$$ which implies that
   $q\Vdash\exists^\infty_{k<\omega}(f_{\eta_k}(m)\leq\dot{g}(m))$. This last fact
   contradicts that $\langle f_{\eta_k}(m)\ /\ k<\omega\rangle$ is increasing.
\end{proof}

\section{A method to construct nice iteration candidates}\label{SecConstr}

In a very general setting, we show how to construct nice iteration candidates. We then apply this method to build the iteration for our main result.

For a nice iteration candidate, any legal stem sequence has to be solved by some sequence of names of ultrafilters. But recall from
 Definition~\ref{DefNiceCandidate}(2) that we want all witnesses $D_\alpha$
 to be in $V^{\Por'_\alpha}$, and in practice this will be a model of size
 $\le \knm$ (the value we want to force for $\non(\Mwf)$).  So we want to
 have as few such sequences of names of ultrafilters as possible, i.e., each sequence should solve
 many  legal stem sequences. For this purpose, we use the following classical result of Engelking and Kar\l owicz, which essentially says that a product of at most $2^\chi$ discrete spaces of size $\chi$ has a dense set of size $\chi$ in an appropriate box topology (in our applications, $\chi$ is between $\kb$ and $\knm$).


\begin{theorem}[Engelking and Kar\l owicz {\cite{TopThm}}, see also {\cite{MoreEngKarl}}]\label{TopTHM}
   Assume $\chi^{<\theta}=\chi$, $\delta<(2^\chi)^+$ an ordinal and $\langle A_\alpha\rangle_{\alpha<\delta}$ a sequence of sets of size $\leq\chi$. Then there is a set $\{h_\epsilon\ /\ \epsilon<\chi\}\subseteq\prod_{\alpha<\delta}A_\alpha$ such that, for any $x\in\prod^{<\theta}_{\alpha<\delta}A_\alpha:=\bigcup_{E\in[\delta]^{<\theta}}\prod_{\alpha\in E}A_\alpha$, there is $\epsilon<\chi$ such that $x\subseteq h_\epsilon$. Moreover, $\prod^{<\theta}_{\alpha<\delta}A_\alpha$ can be partitioned into sets $\langle L^*_\epsilon\rangle_{\epsilon<\chi}$ such that
   \begin{enumerate}[(i)]
     \item if $x\in L^*_\epsilon$ then $x\subseteq h_\epsilon$ and
     \item for all $x,y\in L^*_\epsilon$, $\dom x$ and $\dom y$ have the same order type and the order-preserving isomorphism $g:\dom x\to\dom y$ is the identity on $\dom x\cap \dom y$.
   \end{enumerate}
   When $2^\chi=\chi^+$, we additionally have
   \begin{enumerate}[(i')]
   \setcounter{enumi}{1}
      \item for all $x,y\in L^*_\epsilon$, $\dom x$ and $\dom y$ have the same order type and $\dom x\cap \dom y$ is an initial segment of both $\dom x$ and $\dom y$.
   \end{enumerate}
\end{theorem}

Fix $\kb$ as in section~\ref{SecItCandidates}. Assume $\kb\leq\chi=\chi^{\aleph_0}$, $\delta<(2^\chi)^+$ an ordinal and $\delta=S\cup C$ a disjoint union.
For each $\alpha<\delta$ let $A_\alpha=\omega^{<\omega}\times\omega$ if
$\alpha\in S$, otherwise, $A_\alpha=\kb$. Let $\{h_\epsilon\ /\
\epsilon<\chi\}$ and $\langle L^*_\epsilon\rangle_{\epsilon<\chi}$ be as in
Theorem~\ref{TopTHM} applied to $\theta=\aleph_1$. Therefore, we can partition
$\stemseq_\delta$ into the sets
$\langle\Lambda_\epsilon\rangle_{\epsilon<\chi}$ such that
$\mathbf{x}\in\Lambda_\epsilon$ iff $z_{\mathbf{x}}\in L^*_\epsilon$, where
$\dom z_{\mathbf{x}}=\{\alpha_{\mathbf{x},k,l}\ /\ k<\omega,\
l<l^*_{\mathbf{x}}\}$ and, for $k<\omega$ and $l<l^*_{\mathbf{x}}$, if $l\in
v_{\mathbf{x},S}$ then
$z_{\mathbf{x}}(\alpha_{\mathbf{x},k,l})=(s^*_{\mathbf{x},l},n^*_{\mathbf{x},l})$,
otherwise, $z_{\mathbf{x}}(\alpha_{\mathbf{x},k,l})=\gamma_{\mathbf{x},k,l}$
when $l\in v_{\mathbf{x},C}$.

Here, $h_\epsilon$ is seen as a \emph{guardrail}
for the countable $\Delta$-systems that matches a stem sequence in
$\Lambda_\epsilon$.
All conditions following the same guardrail will be
compatible with each other.
This is because, for an iteration candidate of length
$\delta$ where $S$ corresponds to the coordinates where subposets of $\Eor$ are
used, if $\sqpk$ is a $\Delta$-system that matches
$\mathbf{x}\in\Delta_\epsilon$, the function $h_\epsilon$ describes the
behavior of each $p_k$, that is, if $\zeta\in \supp p_k\menos S$,
$p_k(\zeta)=h_\epsilon(\zeta)$ and, if $\zeta\in S\cap\supp p_k$, then
$h_\epsilon(\zeta)$ tells the stem and the width of the slalom corresponding to
$p_k(\zeta)$. All this information depends only on $\epsilon$ (and the
coordinate $\zeta$) and all the $\Delta$-systems matching stem sequences in
$\Lambda_\epsilon$ are described by the same information.

We show a way to construct, inductively, a nice iteration candidate $\mathbf{q}$ with $\delta_{\mathbf{q}}=\delta$, $S_{\mathbf{q}}=S$ and $C_{\mathbf{q}}=C$ by using the guardrails $\langle h_\epsilon\ /\ \epsilon<\chi\rangle$. Furthermore, we find $\langle\dot{D}^\epsilon_\alpha\ /\ \epsilon<\chi,\ \alpha\leq\delta\rangle$ such that, for each $\epsilon<\chi$, $\mathbf{D}^\epsilon_\delta:=\langle\dot{D}^\epsilon_\alpha\rangle_{\alpha\leq\delta}$ solves all the legal stem sequences of $\Lambda_\epsilon$. 

\subsection*{Induction basis}
When $\delta=0$, choose an arbitrary non-principal ultrafilter $D^\epsilon_0$ for each $\epsilon<\chi$.

\begin{lemma}[Successor step]\label{succstep}
   Assume $\delta=\alpha+1$. Let $\mathbf{q}\frestr\alpha$ be a nice iteration candidate of length $\alpha$ with $S_{\mathbf{q}\upharpoonright\alpha}=S\cap\alpha$ and let $\langle\dot{D}^\epsilon_\xi\ /\ \epsilon<\chi,\ \xi\leq\alpha\rangle$ be such that, for each $\epsilon<\chi$, $\mathbf{D}^\epsilon_\alpha$ solves all $\mathbf{x}\in\stemseq_\alpha\cap\Lambda_\epsilon$ that are legal for $\mathbf{q}\frestr\alpha$. Let $\mathbf{q}$ be an iteration candidate of length $\delta$ that extends $\mathbf{q}\frestr\alpha$ such that the following conditions hold.
   \begin{enumerate}[(i)]
      \item $\alpha\in S_{\mathbf{q}}$ iff $\alpha\in S$.
      \item In the case $\alpha\in S$, $\Por_\alpha$ forces $\dot{D}^\epsilon_\alpha\cap V^{\Por'}_\alpha\in V^{\Por'_\alpha}$ for all $\epsilon<\chi$.
   \end{enumerate}
   \underline{Then}, there are $\Por_{\alpha+1}$-names $\langle\dot{D}^\epsilon_{\alpha+1}\ /\ \epsilon<\chi\rangle$ such that, for each $\epsilon<\chi$, $\mathbf{D}^\epsilon_{\alpha+1}=\mathbf{D}^\epsilon_\alpha\widehat{\ \ }\langle\dot{D}^\epsilon_{\alpha+1}\rangle$ solves all $\mathbf{x}\in\stemseq_{\alpha+1}\cap\Lambda_\epsilon$ that are legal for $\mathbf{q}$.
\end{lemma}
\begin{proof}

   It is enough to prove the following.

\begin{claim}\label{succClaim}
   $\Por_{\alpha+1}$ forces that, for any $\epsilon<\chi$, the family
   $$\dot{D}^\epsilon_\alpha\ \cup\ \big\{  A_{\bar p }  \ /\ \bar p \in P^\infty_{\mathbf{q},\mathbf{x}},\ \mathbf{x}\in\Lambda_\epsilon\cap\stemseq_{\alpha+1}\textrm{\ legal},\ \lim\nolimits_{\mathbf{D}^\epsilon_\alpha}\bar p \in\dot{G} \big\}$$
   (where $A_{\bar p}:=\{ k<\omega\ /\ p_{k}\in\dot{G} \}$ for any
     $\bar p = \sqpk $) has the finite intersection property.
\end{claim}
\begin{proof}
    In the case $\alpha\in C$, it is enough to prove that, if
    $\mathbf{x}\in\Lambda_\epsilon\cap\stemseq_{\alpha+1}$ is legal for
    $\mathbf{q}$, $\sqpk\in
    P^\infty_{\mathbf{q},\mathbf{x}}$ and $q$ is its
    $\mathbf{D}^\epsilon_{\alpha}$-limit, then $q$ forces (with respect to
    $\Por_{\alpha+1}$) that $\{k<\omega\ / \
    p_{k}\in\dot{G}\}\in\dot{D}^\epsilon_\alpha$. We may assume that
    $\alpha\in\supp q$ (if not, $\supp p_{k}\subseteq\alpha$ for all $k<\omega$
    and the claim is straightforward), so
    $p_{k}(\alpha)=q(\alpha)=h_\epsilon(\alpha)$ for all $k<\omega$.
    On the other hand, $q\frestr\alpha$ forces that $\{k<\omega\ /\
    p_{k}\frestr\alpha\in\dot{G}\}\in\dot{D}^\epsilon_\alpha$ (because
    $q\frestr\alpha$ is the $\mathbf{D}^\epsilon_\alpha$-limit of
    $\sqk{p_{k}\frestr\alpha}$), so the conclusion is clear.

    Now, assume that $\alpha\in S$. Let $i^*<\omega$, $\mathbf{x}^i\in\Lambda_\epsilon\cap\stemseq_{\alpha+1}$ legal for $\mathbf{q}$ for $i<i^*$, $\sqk{p_{i,k}} \in P^\infty_{\mathbf{q},\mathbf{x}^i}$, $q_i=\lim_{\mathbf{D}^\epsilon_\alpha}p_{i,k}$, a $\Por_\alpha$-name $\dot{a}$ of a set in $\dot{D}^\epsilon_\alpha$, a fixed $k^*<\omega$ and a condition $r\in\Por_{\alpha+1}$ stronger than $q_i$ for each $i<i^*$. We find an $r'\leq r$ in $\Por_{\alpha+1}$ and a $k>k^*$ such that $r'$ forces that $k\in\dot{a}$ and $p_{i,k}\in\dot{G}$ for all $i<i^*$. We may assume that $r\frestr\alpha$ forces $\dot{a}\subseteq\bigcap_{i<i^*}\{k<\omega\ /\ p_{i,k}\frestr\alpha\in\dot{G}\}\in\dot{D}^\epsilon_\alpha$. Without loss of generality, we assume that $\alpha\in\supp(q_i)$ for all $i<i^*$, so, if $h_\epsilon(\alpha)=(s,n)$, then $p_{i,k}(\alpha)=(s,\dot{\varphi}_{i,k})$ for some $\Por'_\alpha$-name of a slalom $\dot{\varphi}_{i,k}$ of width $n$ and $q_i(\alpha)=(s,\dot{\psi}_i)$ where $\dot{\psi}_i$ is a $\Por_\alpha$-name of the $\dot{D}^\epsilon_\alpha$-limit of $\sqk{\dot{\varphi}_{i,k}}$ (which is forced to be in $V^{\Por'_\alpha}$ by (ii)). Let $G_\alpha$ be $\Por_\alpha$-generic over $V$ with $r\frestr\alpha\in G_\alpha$. In $V_\alpha=V[G_\alpha]$, let $r(\alpha)=(t,\psi')\in\Eor\cap V^{\Por'_\alpha}$, which is stronger than $q_i(\alpha)=(s,\psi_i)$ for all $i<i^*$. As $t(j)\notin\psi_i(j)$ for any $j<|t|$, then the set $a_i=\{k<\omega\ /\ \forall_{j<|t|}(t(j)\notin\varphi_{i,k}(j))\}\in D^\epsilon_\alpha$. So choose $k>k^*$ in $a\cap\bigcap_{i<i^*}a_i$ and let $r'(\alpha)=(t,\psi'')$ where $\psi''(j)=\psi'(j)\cup\bigcup_{i<i^*}\varphi_{i,k}(j)$ for all $j<\omega$. Clearly, $r'(\alpha)$ is stronger than $r(\alpha)$ and than $p_{i,k}(\alpha)$ for all $i<i^*$. Back in $V$, let $r'\frestr\alpha\leq r\frestr\alpha$ in $\Por_\alpha$ forcing the above statement, so $r'=r'\frestr\alpha\cup\{(\alpha,\n{ r'(\alpha)})\}$ is as desired.
\end{proof}

   Choose $\dot{D}^\epsilon_{\alpha+1}$ a $\Por_{\alpha+1}$-name of an ultrafilter containing the set of the claim.
\end{proof}

\begin{lemma}[Limit step]\label{limstep}
   Assume $\delta$ is a limit ordinal, $\mathbf{q}$ is an iteration candidate of length $\delta$ and $\langle\dot{D}^\epsilon_\alpha\ /\ \epsilon<\chi,\ \alpha<\delta\rangle$ a sequence of $\Por_\delta$-names such that, for each $\alpha<\delta$ and $\epsilon<\chi$, $\mathbf{D}^\epsilon_\alpha$ solves all $\mathbf{x}\in\stemseq_\alpha\cap\Lambda_\epsilon$ that are legal for $\mathbf{q}\frestr\alpha$. \underline{Then}, there are $\Por_{\delta}$-names $\langle\dot{D}^\epsilon_{\delta}\ /\ \epsilon<\chi\rangle$ such that, for each $\epsilon<\chi$, $\mathbf{D}^\epsilon_{\delta}=\mathbf{D}^\epsilon_{<\delta}\widehat{\ \ }\langle\dot{D}^\epsilon_{\delta}\rangle$ solves all $\mathbf{x}\in\stemseq_{\delta}\cap\Lambda_\epsilon$ that are legal for $\mathbf{q}$ (here, $\mathbf{D}^\epsilon_{<\delta}=\langle\dot{D}^\epsilon_\alpha\rangle_{\alpha<\delta}$).
\end{lemma}

\begin{proof}
If $\delta$ has uncountable cofinality, let $\dot{D}^\epsilon_\delta$ be a $\Por_\delta$-name of the ultrafilter $\bigcup_{\xi<\delta}\dot{D}^\epsilon_\xi$. So assume that $\delta$ has countable cofinality.

\begin{claim}\label{limitClaim}
    $\Por_{\delta}$ forces that, for any $\epsilon<\chi$, the family
$$\bigcup_{\alpha<\delta}\dot{D}^\epsilon_\alpha\ \cup \
   \big\{ A_{\bar p} \
   /\ \bar p \in
P^\infty_{\mathbf{q},\mathbf{x}},\
\mathbf{x}\in\Lambda_\epsilon\cap\stemseq_\delta\textrm{\ legal},\
\lim\nolimits _{\mathbf{D}^\epsilon_{<\delta}} \bar p \in\dot{G} \big\}$$
 has the finite intersection property.
\end{claim}
\begin{proof}
   Let $\{\mathbf{x}^i\ /\ i<i^*\}$ be a finite subset of $\Lambda_\epsilon\cap\stemseq_\delta$ of legal stem sequences for $\mathbf{q}$, $\sqk{p_{i,k} }\in P^\infty_{\mathbf{q},\mathbf{x}^i}$ for each $i<i^*$, $q_i=\lim_{\mathbf{D}^\epsilon_{<\delta}}p_{i,k}$ and $\dot{a}$ a $\Por_\delta$-name of a set in $\bigcup_{\alpha<\delta}\dot{D}^\epsilon_\alpha$. Let $p\in\Por_\delta$ be a condition stronger than $q_{i}$ for all $i<i^*$ and let $k^*<\omega$ be arbitrary. We want to find $p^*\leq p$ and $k>k^*$ such that $p^*$ is stronger than $p_{i,k}$ for all $i<i^*$ and forces $k\in\dot{a}$.

   As in the notation of Definition~\ref{DefBlueprint}, for each $i<i^*$ let $w_i=w_{\mathbf{x}^i}$ $l^*_i=l^*_{\mathbf{x}^i}$, $v_{i,S}=v_{\mathbf{x}^i,S}$ and so on. For the nontrivial case, we assume that $\sup_{l<l^*_i,k<\omega}\{\alpha_{i,k,l}\}=\delta$. For $i<i^*$, $l\in v_{i,S}$ and $k<\omega$,  $p_{i,k}(\alpha_{i,k,l})=(s^*_{i,l},\dot{\varphi}_{i,k,l})$ where $\dot{\varphi}_{i,k,l}$ is a $\Por'_{\alpha_{i,k,l}}$-name of a slalom of width $n^*_{i,l}$.

   By strengthening $p$, find $\alpha<\delta$ that contains $\supp p\cup\bigcup_{i<i^*}\supp q_{i}$ and such that $p$ forces that $\dot{a}\in\dot{D}^\epsilon_\alpha$, so, without loss of generality, $\dot{a}$ can be assumed to be a $\Por_\alpha$-name. Even more, $\alpha<\delta$ can be found so that, for any $i<i^*$ and $l<l^*_i$, if there is some $k<\omega$ such that $\alpha_{i,k,l}\geq\alpha$, then $\sup_{k<\omega}\{\alpha_{i,k,l}\}=\delta$. For each $i<i^*$, let $u_i=\{l<l^*_i\ /\ \sup_{k<\omega}\{\alpha_{i,k,l}\}=\delta\}$ (note that this is an interval of the form $[l'_i,l^*_i)$ where $l'_i$ is above the root $\Delta_{\mathbf{x}^i}=\supp q_i$).

   By hypothesis, find $p'\leq p$ in $\Por_\alpha$ and $k>k^*$ such that $p'$ is stronger than $p_{i,k}\frestr\alpha$ for all $i<i^*$ and forces $k\in\dot{a}$. Moreover, $k$ can be found so that\footnote{This is the only place where we need $\langle\alpha_{i,k,l}\ /\ k<\omega\rangle$ increasing.} $\alpha_{i,k,l}>\alpha$ for any $i<i^*$ and $l\in u_i$. Thus, because $\Eor$ is $\sigma$-centered and each $z_{\mathbf{x}^i}\subseteq h_\epsilon$, there is a condition $p^*\leq p'$ in $\Por_\delta$ stronger than $p_{i,k}$ for all $i<i^*$. Indeed, $\supp p^*=\supp p'\cup\bigcup_{i<i^*}\supp p_{i,k}$, $p^*(\zeta)=p'(\zeta)$ for $\zeta\in\supp p'$, $p^*(\zeta)=p_{i,k}(\zeta)=h_\epsilon(\zeta)$ for $\zeta\in\supp p_{i,k}\cap v_{i,C}\menos\alpha$ and $p^*(\zeta)=(s_\zeta,\dot{\psi}_\zeta)$ for $\zeta\in\big(\bigcup_{i<i^*}\supp p_{i,k}\cap v_{i,S}\big)\menos\alpha$ where $h_\epsilon(\zeta)=(s_\zeta,n_\zeta)$ and $\dot{\psi}_\zeta$ is a $\Por'_\zeta$-name of the slalom given by $\dot{\psi}(j)=\bigcup\{\dot{\varphi}_{i,k,l}(j)\ /\ \alpha_{i,k,l}=\zeta,\ l<l^*_i,\ i<i^*\}$. $p^*$ is as desired because, if $\zeta=\alpha_{i,k,l}=\alpha_{i',k,l'}$, when $\zeta\in C$ then $p_{i,k}(\zeta)=p_{i',k}(\zeta)=h_{\epsilon}(\zeta)$ and, when $\zeta\in S$, $p_{i,k}(\zeta)=(s^*_{i,l},\dot{\varphi}_{i,k,l})$, $p_{i',k}(\zeta)=(s^*_{i',l'},\dot{\varphi}_{i',k,l'})$ and $s^*_{i,l}=s^*_{i',l'}=s_\zeta$.
\end{proof}

   Choose $\dot{D}^\epsilon_\delta$ a $\Por_\delta$-name of an ultrafilter that contains the set of the previous claim.
\end{proof}

\begin{remark}\label{RemWhyEK}
   In Lemma \ref{succstep}, if all the $\dot{D}^\epsilon_\alpha$
   ($\epsilon<\chi$) are (forced to be) equal to some ultrafilter
   $\dot{D}_\alpha$, then Claim~\ref{succClaim} can be similarly proven without
   fixing $\epsilon$, that is, $\Por_{\alpha+1}$ forces that
   $\dot{D}_\alpha\ \cup\ \big\{ A_{\bar p} \ /\
    \bar p \in
   P^\infty_{\mathbf{q},\mathbf{x}},\
   \mathbf{x}\in\stemseq_{\alpha+1}\textrm{\ legal},\
   \lim\nolimits _{\mathbf{D}^\epsilon_{\alpha}} \bar p \in\dot{G} \big\}$ has the
   finite intersection property. Nevertheless, in Lemma \ref{limstep} when $\delta$ is a limit of
   countable cofinality, the corresponding statement for Claim~\ref{limitClaim}
   may not be true when all the $\dot{D}^\epsilon_\alpha$ are the same for each
   $\alpha<\delta$ so, at that point, it becomes necessary to have different
   sequences of names of ultrafilters for each $\epsilon<\chi$ and
   Theorem~\ref{TopTHM} must be used to have as few sequences as possible (each
   one with respect to a guardrail $h_\beta$).

   For instance, let $\delta=\omega$ and $\epsilon,\epsilon'<\chi$ such that $h_\epsilon$ and $h_{\epsilon'}$ are incompatible everywhere, that is, for each $m<\omega$, if $A_m=\kb$ then $\Vdash_{m}h_\epsilon(m) \perp_{\Qnm_m} h_{\epsilon'}(m)$ and, when $A_m=\omega^{<\omega}\times\omega$, the first coordinates of both $h_\epsilon(m)$ and $h_{\epsilon'}(m)$ are incompatible. If $\mathbf{x}\in\Delta_\epsilon\cap\stemseq_\omega$ and $\mathbf{x}'\in\Delta_{\epsilon'}\cap\stemseq_\omega$ are legal for $\mathbf{q}$ such that $l^*_{\mathbf{x}}=l^*_{\mathbf{x}'}=1$ and $\alpha_{\mathbf{x},k,0}=\alpha_{\mathbf{x}',k,0}=k$, if $\sqpk\in P^\infty_{\mathbf{q},\mathbf{x}}$ and $\sqk{p'_k}\in P^\infty_{\mathbf{q},\mathbf{x}'}$, then $\lim_{\mathbf{D}^\epsilon_{<\omega}}p_k=\lim_{\mathbf{D}^{\epsilon'}_{<\omega}}p'_k$ is the trivial condition and it is clear that $\Por_\omega$ forces $\{k<\omega\ /\ p_k\in\dot{G}\}\cap\{k<\omega\ /\ p'_k\in\dot{G}\}=\varnothing$.
\end{remark}

\section{Proof of the main result}\label{SecMain}

To prove our main result, we construct a nice iteration candidate with the book-keeping arguments described in Subsection~\ref{SubsecIt}. Thanks to Main Lemma~\ref{PresUnbLemma}, we can guarantee that $\bfrak$ is the value we want in the extension.

\newcommand{\thetanull}{\kappa_{an}}
\newcommand{\thetaeins}{\kappa_{cn}}
\newcommand{\nunu}{\kappa_{nm}}
\newcommand{\laam}{\kappa_{ct}}

\begin{maintheorem}\label{MainThm}
   Let $\thetanull\leq\thetaeins\leq\kb\leq\nunu=\nunu^{\aleph_0}$ be regular uncountable cardinals and $\laam=\laam^{<\nunu}\leq2^\chi$ where $\kb\leq\chi=\chi^{\aleph_0}\leq\nunu<\laam$. Assume $\bfrak=\dfrak=\kb$. Then, there is a ccc poset that forces
$$\add(\Nwf)=\thetanull\leq\cov(\Nwf)=\thetaeins\leq\bfrak=\kb\leq\non(\Mwf)=\nunu<\cov(\Mwf)=\cfrak=\laam.$$
\end{maintheorem}

Note that, assuming GCH, if $\thetanull\leq\thetaeins\leq\kb\leq\nunu$ are regular uncountable cardinals and $\laam>\nunu$ is a cardinal of cofinality $\geq\nunu$, it is not hard to construct a model, by forcing, that satisfies the hypothesis of the theorem with $\chi=\kb$.

\begin{proof}
   We construct a nice iteration candidate $\mathbf{q}$ of length $\delta_{\mathbf{q}}=\laam\cdot\nunu$ (ordinal product) that forces our desired statement. Let $\laam=S'\cup C'$ be a partition where each $S'$ and $C'$ has size $\laam$ and also let $C'=C'_0\cup C'_1\cup C'_2\cup C'_3$ be a partition where each $C'_i$ has size $\laam$. Let $S=S_{\mathbf{q}}=\bigcup_{\beta<\nunu}(\laam\cdot\beta+S')$, $C_i=\bigcup_{\beta<\nunu}(\laam\cdot\beta+C'_i)$ for $i<4$ and $C=C_{\mathbf{q}}=C_0\cup C_1\cup C_2\cup C_3$.

   We construct $\mathbf{q}$ by recursion using the method in Section~\ref{SecConstr}. The induction basis and the limit step are clear by Lemma \ref{limstep}, so we only have to describe what we do in the successor step in such a way that Lemma \ref{succstep} can be applied. Assume we have constructed our iteration up to $\alpha<\laam\cdot\nunu$ and that $\langle\dot{D}^\epsilon_\xi\ /\ \epsilon<\chi,\ \xi\leq\alpha\rangle$ is as in Lemma \ref{succstep}. $\alpha$ is of the form $\laam\cdot\beta+\zeta$ for some (unique) $\beta<\nunu$ and $\zeta<\laam$. Consider:
   \begin{enumerate}[(i)]
      \item $\{\Anm_{\beta,\xi}\ /\ \xi\in C'_0\}$ lists the $\Por_{\laam\cdot\beta}$-names of all ccc posets whose domain is an ordinal $<\thetanull$.
      \item $\{\Bnm_{\beta,\xi}\ /\ \xi\in C'_1\}$ lists the $\Por_{\laam\cdot\beta}$-names of all subalgebras of random forcing $\Bor^{V^{\Por_{\laam\cdot\beta}}}$ of size $<\thetaeins$.
      \item $\{\Dnm_{\beta,\xi}\ /\ \xi\in C'_2\}$ lists the $\Por_{\laam\cdot\beta}$-names of all $\sigma$-centered subposets of Hechler forcing $\Dor^{V^{\Por_{\laam\cdot\beta}}}$ of size $<\kb$.
      \item $\{\dot{F}^\beta_\xi\ /\ \xi\in S'\}$ lists the $\Por_{\laam\cdot\beta}$-names for all sets of size $<\nunu$ of slaloms of finite width in $V^{\Por_{\laam\cdot\beta}}$.
   \end{enumerate}
   If $\alpha\in C$, let
   \[\Qnm_\alpha=\left\{
        \begin{array}{ll}
            \Anm_{\beta,\zeta} & \textrm{if $\zeta\in C'_0$,}\\
            \Bnm_{\beta,\zeta} & \textrm{if $\zeta\in C'_1$,}\\
            \Dnm_{\beta,\zeta} & \textrm{if $\zeta\in C'_2$,}\\
            \Cnm & \textrm{(Cohen forcing) if $\zeta\in C'_3$.}
        \end{array}\right.\]
   If $\alpha\in S$, we need to construct $\Por'_\alpha$ and we
want\footnote{
See Remark~\ref{less.or.le}, which explains why we only require
$|\Por'_\alpha|\le \nunu$ rather than the strict inequality that the
reader might have expected.}
it to have size $\leq\nunu$. Given $\epsilon<\chi$ and $\dot{a}$ a (nice) $\Por_\alpha$-name of a subset of $\omega$, choose a maximal antichain $A^\epsilon_{\dot{a}}$ that decides either $\dot{a}\in\dot{D}^\epsilon_\alpha$ or $\omega\menos\dot{a}\in\dot{D}^\epsilon_\alpha$. Therefore, by closing under this and other simpler operations, we can find $\Por'_\alpha\lessdot\Por_\alpha$ of size $\leq\nunu$ such that $\dot{F}_\zeta$ is a $\Por'_\alpha$-name and, for any $\epsilon<\chi$ and a (nice) $\Por'_\alpha$-name $\dot{a}$ of a subset of $\omega$, $A^\epsilon_{\dot{a}}\subseteq\Por'_\alpha$ (because $\nunu^{\aleph_0}=\nunu$), which implies that there is a $\Por'_\alpha$-name of $\dot{D}^\epsilon_\alpha\cap V^{\Por'_\alpha}$. Let $\Qnm_\alpha=\Eor^{V^{\Por'_\alpha}}$, which adds a real that evades eventually all the slaloms from $\dot{F}_{\zeta}$. It is clear that Lemma \ref{succstep} applies, which finishes the construction.

\relax   From the results already proved or cited, it is easy to check that
$\Por_\delta$ forces $\thetanull\leq\add(\Nwf)$, $\thetaeins\leq\cov(\Nwf)$,
$\kb\leq\bfrak$ and $\nunu\leq\non(\Mwf)$.  The relations
 $\add(\Nwf)\leq\thetanull$ and
$\cov(\Nwf)\leq\thetaeins$ in the extension are consequences of Corollary \ref{larged}(a) applied to the pairs $(\theta,\sqsubset)=(\kan,\in^*)$ (see Example \ref{SubsecUnbd}(4)) and $(\theta,\sqsubset)=(\kcn,\pitchfork)$ (see Example \ref{SubsecUnbd}(3)), respectively. The crucial inequality
 $\bfrak\leq\kb$ is a
consequence of Lemma~\ref{PresUnbLemma} (applied to a scale $\langle f_\alpha\rangle_{\alpha<\kb}$ that lives the ground model, which exists because, there, $\bfrak=\dfrak=\kb$). Besides,
 $\non(\Mwf)\leq\nunu$ holds in the
extension because of the $\nunu$-cofinal many Cohen reals added along the
iteration. It is also clear that $\cfrak\leq\laam$ is forced.

  Finally, by Corollary \ref{larged}(b) applied to the pair $(\theta,\sqsubset)=(\knm^+,\eqcirc)$ (see Example \ref{SubsecUnbd}(1)), $\cov(\Mwf)=\dfrak_\eqcirc\geq\kct$.
\end{proof}

\begin{remark}
\label{less.or.le}
   If we further assume that $\chi<\nunu$ and $\mu^{\aleph_0}<\nunu$ for all $\mu<\nunu$, then we can similarly construct a nice iteration candidate of length $\laam$ forcing the same as in Theorem~\ref{MainThm}. In the argument of this, for $\alpha\in S$, we can construct $\Por'_\alpha$ of size $<\nunu$, so we can force $\non(\Mwf)\leq\nunu$ by Corollary \ref{larged}(a) applied to $(\theta,\sqsubset)=(\knm,\eqcirc)$. However, this is less general because $\nunu$ is not allowed to be a successor of a cardinal with countable cofinality.
\end{remark}

\begin{remark}
   A similar proof of Theorem~\ref{MainThm} can be performed using bounded versions of $\Eor$ to ease the compactness arguments, but it has the disadvantage that we are restricted to $2^\chi=\chi^+=\laam$ and $\chi=\nunu$. The argument is similar but much more difficult, we point out the differences with the presented argument.
   \begin{enumerate}[(1)]
      \item Fix $b:\omega\to\omega$ a fast increasing function with $b(0)>0$. Let $\Eor^b$ be the poset whose conditions are pairs $(s,\varphi)$ where $s$ is a finite sequence below $b$ and $\varphi$ is a slalom of width at most $|s|$. The order is similar to $\Eor$. Like $\Eor$, this poset adds an eventually different real (below $b$) and does not add dominating reals (moreover, it is $\leq^*$-good), however, it is not $\sigma$-centered.
      \item In all the arguments, everything related to $\Eor$ should be respectively modified to the context of $\Eor^b$.
      \item In Definition~\ref{DefBlueprint}, we additionally have to include Borel functions that code the names of slaloms corresponding to the coordinates of subposets of $\Eor^b$ of the conditions of the countable $\Delta$-system that is coded. In this case, those codes should be called \emph{blueprints}. Moreover, $n^*_l\leq|s^*_l|$ for all $l<l^*$.
      \item $2^\chi=\chi^+$ is assumed because we need (ii') of Theorem~\ref{TopTHM} in this case. Guardrails $h_\beta$ should also talk about the Borel functions included in the blueprints of $\Lambda_\epsilon$ (by further assuming $\cfrak=\kb$ in the ground model), so the last part of the proof of Claim~\ref{limitClaim} could be argued.
      \item In the construction of the iteration for the main result, we have to guarantee that the used subposets of $\Eor^b$ don't add random reals nor destroy a witness of $\add(\Nwf)$ that we want to preserve, that is, that they are both $\kan$-$\in^*$-good and $\kcn$-$\pitchfork$-good. For this, a notion of $(\pi,\rho)$-linkedness, defined in \cite{KO}, justifies the desired preservation (for $\kan$-$\in^*$-goodness, see \cite[Section~5]{BrMe}).
   \end{enumerate}
\end{remark}

\section{Questions}\label{SecQ}

\begin{question}\label{QwithCovM<c}
   Is there a model of $\aleph_1<\add(\Nwf)<\cov(\Nwf)<\bfrak<\non(\Mwf)<\cov(\Mwf)<\cfrak$? or just a model of $\bfrak<\non(\Mwf)<\cov(\Mwf)<\cfrak$?
\end{question}

A ZFC model of $\aleph_1<\add(\Nwf)=\non(\Mwf)<\cov(\Mwf)=\cf(\Nwf)<\cfrak$ was constructed in \cite[Thm. 11]{Me-MatIt} by a matrix iteration (a technique to construct fsi's in a non-trivial way). The difficulty to answer Question \ref{QwithCovM<c} lies in the fact that there are no known easy fsi constructions that force $\non(\Mwf)<\cov(\Mwf)<\cfrak$.

\begin{question}
   Is there a model of $\aleph_1<\add(\Nwf)<\bfrak<\cov(\Nwf)<\non(\Mwf)<\cov(\Mwf)<\cfrak$? or just a model of $\bfrak<\cov(\Nwf)<\non(\Mwf)$?
\end{question}

As pointed out by Judah and Shelah \cite{JShDOm} (see also \cite{Paw-Dom}), subalgebras of random forcing may add dominating reals, so there are similar difficulties as those described in Section \ref{SecOverview} for subposets of $\Eor$. It seems that sophisticated techniques as in \cite{ShCov} may help to deal with this problem.

\bibliography{left}

\begin{thebibliography}{FGKS15}

\bibitem[AY08]{MoreEngKarl}
Uri Abraham and Yimu Yin.
\newblock A note on the {E}ngelking-{K}ar\l owicz theorem.
\newblock {\em Acta Math. Hungar.}, 120(4):391--404, 2008.

\bibitem[BJ95]{BaJu95}
Tomek Bartoszy\'nski and Haim Judah.
\newblock {\em {Set Theory: On the Structure of the Real Line}}.
\newblock A K Peters, Wellesley, Massachusetts, 1995.

\bibitem[BM14]{BrMe}
J{\"o}rg Brendle and Diego~Alejandro Mej{\'{\i}}a.
\newblock Rothberger gaps in fragmented ideals.
\newblock {\em Fund. Math.}, 227(1):35--68, 2014.

\bibitem[Bre91]{Br-Cichon}
J{\"o}rg Brendle.
\newblock Larger cardinals in {C}icho\'n's diagram.
\newblock {\em J. Symbolic Logic}, 56(3):795--810, 1991.

\bibitem[CKP85]{CKP}
Jacek Cicho{\'n}, Anastasis Kamburelis, and Janusz Pawlikowski.
\newblock On dense subsets of the measure algebra.
\newblock {\em Proc. Amer. Math. Soc.}, 94(1):142--146, 1985.

\bibitem[EK65]{TopThm}
Ryszard Engelking and Monika Kar{\l}owicz.
\newblock Some theorems of set theory and their topological consequences.
\newblock {\em Fund. Math.}, 57:275--285, 1965.

\bibitem[FGKS15]{FGKS}
Arthur Fischer, Martin Goldstern, Jakob Kellner, and Saharon Shelah.
\newblock Creature forcing and five cardinal characteristics of the continuum.
\newblock {\em Arch. Math. Logic}, to appear, 2015.

\bibitem[{Fre}84]{FremCich}
David~H. {Fremlin}.
\newblock {Cicho\'n's diagram.}
\newblock {\em {Publ. Math. Univ. Pierre Marie Curie, S\'emin. Initiation
  Anal.}}, 66, 1984.

\bibitem[Gol93]{gold}
Martin Goldstern.
\newblock Tools for your forcing construction.
\newblock In {\em Set theory of the reals ({R}amat {G}an, 1991)}, volume~6 of
  {\em Israel Math. Conf. Proc.}, pages 305--360. Bar-Ilan Univ., Ramat Gan,
  1993.

\bibitem[JS90]{JuSh-KunenMillerchart}
Haim Judah and Saharon Shelah.
\newblock The {K}unen-{M}iller chart ({L}ebesgue measure, the {B}aire property,
  {L}aver reals and preservation theorems for forcing).
\newblock {\em J. Symbolic Logic}, 55(3):909--927, 1990.

\bibitem[JS93]{JShDOm}
Haim Judah and Saharon Shelah.
\newblock Adding dominating reals with the random algebra.
\newblock {\em Proc. Amer. Math. Soc.}, 119(1):267--273, 1993.

\bibitem[Kam89]{kamburelis}
Anastasis Kamburelis.
\newblock Iterations of {B}oolean algebras with measure.
\newblock {\em Arch. Math. Logic}, 29(1):21--28, 1989.

\bibitem[KO14]{KO}
Shizuo Kamo and Noboru Osuga.
\newblock Many different covering numbers of {Y}orioka's ideals.
\newblock {\em Arch. Math. Logic}, 53(1-2):43--56, 2014.

\bibitem[Mej13]{Me-MatIt}
Diego~Alejandro Mej{\'{\i}}a.
\newblock Matrix iterations and {C}ichon's diagram.
\newblock {\em Arch. Math. Logic}, 52(3-4):261--278, 2013.

\bibitem[Mil81]{Mi}
Arnold~W. Miller.
\newblock Some properties of measure and category.
\newblock {\em Trans. Amer. Math. Soc.}, 266(1):93--114, 1981.

\bibitem[Paw92]{Paw-Dom}
Janusz Pawlikowski.
\newblock Adding dominating reals with {$\omega^\omega$} bounding posets.
\newblock {\em J. Symbolic Logic}, 57(2):540--547, 1992.

\bibitem[She00]{ShCov}
Saharon Shelah.
\newblock Covering of the null ideal may have countable cofinality.
\newblock {\em Fund. Math.}, 166(1-2):109--136, 2000.
\newblock Saharon Shelah's anniversary issue.

\end{thebibliography}
\bibliographystyle{alpha}





\end{document}